\documentclass[11pt,a4paper]{article}

\usepackage[utf8]{inputenc}
\usepackage{amsmath}
\usepackage{amssymb}
\usepackage{amsthm}
\usepackage{hyperref}

\usepackage{orcidlink}
\usepackage{authblk}

\newcommand{\T}{\mathbb{T}}

\def\sub{\subseteq}

\def\LebT{\tau}

\def\lpar{\left(}
\def\rpar{\right)}
\def\into{\longrightarrow}
\def\sub{\subseteq}
\def\setdiff{\backslash}

\def\klim{{\displaystyle\lim_{k\rightarrow\infty}}}

\def\Bbar{\overline{B}}

 \newtheorem{theorem}{Theorem}[section]
 \newtheorem{lemma}[theorem]{Lemma}
  \newtheorem{proposition}[theorem]{Proposition}
 \newtheorem{corollary}[theorem]{Corollary}

\title{Sufficient Conditions for Strong Natural Boundaries involving Lacunary Sequences}
\author{Rafael Reno S. Cantuba\thanks{Associate Professor, Department of Mathematics and Statistics, De La Salle University (DLSU), Taft Ave., Manila, Philippines, supported by a grant from the Reserach Grants Management Office of DLSU, grant no.: 02FR1TAY24-1TAY25, email: rafael.cantuba@dlsu.edu.ph}\\ ORCID: 0000-0002-4685-8761\orcidlink{0000-0002-4685-8761}}
\date{}

\begin{document}

\maketitle

\begin{abstract} The notion of a regular point is extended to that of a weak regular point, and so the notion of a strong natural boundary may be described locally as the absence of weak regular points, just as a natural boundary is locally the absence of regular points. The main result is a sufficient condition for a strong natural boundary, which is the conjunction of the unboundedness of the coefficients of the underlying power series with either of two specific gap density conditions imposed on the exponents that yields a lacunary sequence.
\end{abstract}

\vspace{0.5em}
\noindent
\begin{quote}
    
    \vspace{0.5em}
\end{quote}

\section{Introduction}

From at least two important notions in analysis, the motivation for this study is drawn, the first of which, is said to be the opposite of the phenomenon of functions analytic in the whole complex plane: that the functions have no analytic continuation beyond the domain of validity of their original definition \cite[p.~189]{seg08}. This leads us to the gap theorems of complex analysis. The connection to harmonic analysis is made through the power series representations involved, and in particular, the integer exponents in these series. If these integer sequences involve gaps, then we have the notion of lacunary sequence, of which, mathematicians are said to have been long fascinated \cite[p.~1]{gra13}.

A dichotomy is present in the study of these gap theorems in the sense that some theorems are about conditions on the exponents of the underlying power series, while others involve conditions on the coefficients. See, for instance, \cite[Chapter~6]{seg08} or the book \cite{sha68}. Even in some extensions of these gap theorems, the dichotomy persists, such as the extension \cite[Proposition~10.4.7]{gra13} of the Hadamard Gap Theorem, that was done through harmonic analysis, which involves only exponents, and the extension in \cite{bre11} of natural boundaries in which the results deal with conditions only on the coefficients. The latter has been further extended to probabilistic methods \cite{dos26}, but still involving only coefficients for the sufficient conditions. In this work, we prove sufficient conditions for gap theorems that involve both the exponents and coefficients.

The notion of (classical) natural boundary in complex analysis is related to the local notion of a (classical) regular point. Analogous to this, we devise a similar construct, which is that of a weak regular point (defined in Section~\ref{PrelimSec}) that is related to a strong natural boundary. A strong natural boundary is a natural boundary, and a regular point is a weak regular point. In Section~\ref{CoeffSec}, we give an upper bound on the coefficients of the power series in question that depends on the length of some interval and the Fourier-Stieltjes transform of a Radon measure constructed using the said interval. This upper bound is then generalized into an absolute upper bound if all elements of the unit circle are weak regular points of the series under consideration. The Nazarov-Tur\'an Lemma is used in Section~\ref{TuranSec} to show that the integral of the absolute value of the analytic function in question along an arbitrary (scaled) arc of the unit circle is bounded above by the corresponding integral for a fixed arc, but with a factor that may still increase without bound. This turns out to be a very important lemma since, when the gap conditions on the exponents of the power series are introduced in Section~\ref{GapSec}, the said factor stabilizes into a constant, and so, the series exhibits the said bound in all directions around the unit circle. This is specifically for the case when the gap condition is that of having zero density. For the other gap condition, we make use of interpolation by point mass measures. Finally, we make use of all the important lemmas and state in Section~\ref{ThmSec} our main result, which is that a sufficient condition for a strong natural boundary is the conjunction of two conditions: that of the coefficients forming an unbounded sequence, and either zero harmonic density or zero density of the exponents. Some examples are mentioned after to show how only one of these conditions is not sufficient for a strong natural boundary. We propose further directions that involve the theory of Hardy spaces and Smirnov spaces.

\section{Preliminaries}\label{PrelimSec}

Let $(a_k)$ be a sequence of complex numbers and let $(n_k)$ be a strictly increasing sequence of positive integers. Each $n_k$, because of some isomorphism of topological groups, may also be viewed as the character $e^{i\theta}\mapsto e^{in_k\theta}$ of the compact group $\T$ of all complex numbers with modulus $1$. Consider a power series $f(z)=\sum_{k=1}^\infty a_kz^{n_k}$ defined if $|z|<1$.

Given $\varepsilon>0$ and a complex number $c$, the open (respectively, closed) disk with radius $\varepsilon$ centered at $c$ shall be denoted by $B(c,\varepsilon)$ (respectively, $\Bbar(c,\varepsilon)$). By an ``arc'' on or along $\T$ we shall mean the intersection of $\T$ with $B(e^{i\alpha},\varepsilon)$ for some $e^{i\alpha}\in\T$ and some $\varepsilon>0$. 
If there exists an arc $I=\T\cap B(e^{i\alpha},\varepsilon)$ such that
\begin{equation}
\sup_{0<r<1}\int_{e^{i\theta}\in I}|f(re^{i\theta})|\ \frac{d\theta}{2\pi}<\infty,\label{Mdef}
\end{equation}
then we shall refer to $e^{i\alpha}\in I\sub T$ as a \emph{weak regular  point} of $f$. The arc $\T\cap B(e^{i\alpha},\varepsilon)$ and the neighborhood $B(e^{i\alpha},\varepsilon)$ of the weak regular point $e^{i\alpha}$ shall be relevant in nearly all succeeding arguments. Also, if $e^{i\alpha}$ is indeed a weak regular point of $f$, then the function $z\mapsto f(e^{i\alpha}z)$ has $1$ as a weak regular point\footnote{A transformation of the same kind may be used to handle the general case when the power series is defined if $|z|<R$, for some $R>0$. Thus, in this paper, like in one of the main references, which is \cite{bre11}, the development of the  theory starts with an assumption of the power series converging if $|z|<1$.}, and for simplicity, we assume henceforth that if $f$ has a weak regular point at $\T$, then this weak regular point is $1$. The set $I$ in this case is the image $I_\delta$ under $s\mapsto e^{is}$ of some ``symmetric'' interval $(-\delta,\delta)$ for some $\delta>0$, and the radius of the neighborhood $B_\delta:=B(1,\varepsilon)$ of $1$ may be routinely verified to be $\varepsilon=2\left|\sin\frac{\delta}{2}\right|$. Also, \eqref{Mdef} in this case becomes
\begin{equation}
\sup_{0<r<1}\int_{-\delta}^\delta|f(re^{i\theta})|\ \frac{d\theta}{2\pi}<\infty.\label{Mdef2}
\end{equation}

A (classical) regular point of $f$ on $\T$ is an element $e^{i\alpha}$ of $\T$ for which there exists $\delta>0$ such that $f$ has an analytic continuation on $B(e^{i\alpha},\delta)$. Similar to the reasoning done earlier for a weak regular point, if a regular point is needed in an argument, we shall assume without loss of generality that such a regular point is $1$.

\begin{proposition} A regular point is a weak regular point.
\end{proposition}
\begin{proof}
If $1$ is a regular point of $f$, then $f$ has an analytic continuation on $B(1,2|\sin\delta|)$ for some $\delta>0$. Without loss of generality, we further assume that $0<\delta<1<\frac{\pi}{2}$ so that $\Bbar(1,2|\sin\frac{\delta}{2}|)\sub B(1,2|\sin\delta|)$. Thus, $f$ is analytic and hence continuous on $\Bbar(1,2|\sin\frac{\delta}{2}|)$. From $0<\delta<1$, we obtain $0<1-\delta<1$. Since $1$ is in the closure of $(0,1)$, there exists $R\in(0,1)$ such that $0<1-\delta<R<1$. Since $f$ is analytic on $B(0,1)$, we find that $f$ is continuous on $\Bbar(0,R)\sub B(0,1)$. If $C:=\{re^{i\theta}\  :\  0<r<1,\  -\delta\leq\theta\leq\delta\}$, then $C\sub\Bbar(1,2|\sin\frac{\delta}{2}|)\cup\Bbar(0,R)$, where the set in the right-hand side is a compact set on which $f$ has just been shown to be continuous. Thus, the supremum in 
\begin{flalign}
    && \int_{-\delta}^\delta |f(re^{i\theta})|\frac{d\theta}{2\pi} &\leq\sup_{re^{i\theta}\in C}|f(re^{i\theta})|\int_{-\delta}^\delta\frac{d\theta}{2\pi}=\frac{\delta}{\pi}\sup_{re^{i\theta}\in C}|f(re^{i\theta})|,& (0<r<1),\nonumber
\end{flalign}
is finite. We obtain \eqref{Mdef2}, and this completes the proof.  
\end{proof}
In other terminology, if by $f$ having $\T$ as a \emph{strong natural boundary} we mean that every element of $\T$ is not a weak regular point of $f$, then $\T$ is a (classical) natural boundary of $f$ (meaning every element of $\T$ is not a regular point of $f$). The notion of strong natural boundary was introduced in \cite[p.~4905]{bre11} for the case when the coefficient sequence $(a_k)$ of $f$ is a bounded sequence of complex numbers. In this work, we do not make such an assumption, as the results are about sufficient conditions for $\T$ to be a strong natural boundary that involve the boundedness (or lack thereof) of the power series coefficients together with some gap assumptions on the exponents $n_k$. To the best of our knowledge, in most of the literature on classical natural boundaries, these two categories of conditions are treated separately, and in this work, we show that a combination of these two types of conditions involve strong natural boundaries. In particular, we give a discussion in the next section about weak regular points and the aforementioned condition on coefficients.

\section{Boundedness of coefficients}\label{CoeffSec}

For the weak regular point $1$ of $f$, the image $U(1)$ of a fixed subinterval $(a,b) \subseteq (-\delta,\delta) $ under the function $t \mapsto e^{it}$ is strictly within the arc $I_\delta$. Let $\chi_{U(1)}$ denote the characteristic function or (set) indicator function of $U(1)$. For a fixed $r \in (0,1)$ and $x \in [-\pi, \pi)$, the function $e^{i\theta} \mapsto f(re^{i(\theta+x)})$ is in the function space $C(\T)$ of all complex-valued continuous functions on $\T$, and if we denote Lebesgue measure on $\T$ by $\LebT$, then the function $e^{i\theta}\mapsto f(re^{i(\theta+x)})\chi_{U(1)}(e^{i\theta})$ is in $L^{1}(\LebT)$. By one form of the Riesz Representation Theorem, there exists a Radon measure $\iota_x$ in the Banach space $M(\T)$, the topological dual of $C(\T)$, such that $g\in C(\T)$ implies $\int_\T g\  d\iota_x= \int_{-\pi}^\pi g(e^{i\theta})f(re^{i(\theta+x)})\chi_{U(1)}(e^{i\theta})\  \frac{d\theta}{2\pi}$. The Fourier-Stieltjes transform of $\iota_x$ is the complex-valued function $\widehat{\iota_x}$ on the dual of the compact group $\T$ defined by $\widehat{\iota_x}(n_k) = \int_{\T} e^{-in_k\theta} d\iota_x(e^{i\theta})$.

\begin{proposition}\label{coeffProp}
If $1$ is a weak regular point of $f$, then for each positive integer $k$, 
\[
a_k = \frac{2\pi}{r^{n_k}(b-a)} \int_{-\pi}^\pi \widehat{\iota_x}(n_k) e^{-in_k x} \frac{dx}{2\pi}.
\]
\end{proposition}

\begin{proof}
By the definition of the measure $\iota_x$ and of Fourier-Stieltjes transform,
\begin{equation}
\widehat{\iota_x}(n_k) = \int_{a}^{b} f(re^{i(\theta+x)})e^{-in_k\theta}\ \frac{d\theta}{2\pi}.\label{FS_trans}
\end{equation}
Multiplying both sides of \eqref{FS_trans} by $e^{-in_k x}$ and integrating with respect to $x$ over $\T$,
\[
\int_{-\pi}^\pi \widehat{\iota_x}(n_k) e^{-in_k x} \frac{dx}{2\pi} = \int_{-\pi}^\pi \left( \int_{a}^{b} f(re^{i(\theta+x)}) e^{-in_k\theta} \ \frac{d\theta}{2\pi} \right) e^{-in_k x} \frac{dx}{2\pi},
\]
 to which we substitute the power series expansion for $f(re^{i(\theta+x)})$, which results to
\[
\int_{-\pi}^\pi \widehat{\iota_x}(n_k) e^{-in_k x} \frac{dx}{2\pi} = \int_{-\pi}^\pi \left( \int_{a}^{b} \sum_{m=1}^\infty a_m r^{n_m} e^{in_m(\theta+x)} e^{-in_k\theta} \ \frac{d\theta}{2\pi} \right) e^{-in_k x} \frac{dx}{2\pi}.
\]
Since $r \in (0,1)$ is fixed and $f$ is analytic on the open disk $B(0,1)$, the power series expansion $\sum_{m=1}^\infty a_m z^{n_m}$ converges absolutely and uniformly on the (compact) boundary of the disk $B(0,r)\sub B(0,1)$. Consequently, the series for $f(re^{i(\theta+x)})$ converges absolutely and uniformly with respect to $(\theta, x) \in [a,b] \times [-\pi, \pi)$. Due to this uniform convergence, term-by-term integration is justified, and Fubini's theorem may be used to obtain
\begin{align*}
\int_{-\pi}^\pi \widehat{\iota_x}(n_k) e^{-in_k x} \frac{dx}{2\pi} &= \sum_{m=1}^\infty a_m r^{n_m} \int_{-\pi}^\pi \left( \int_{a}^{b} e^{in_m\theta} e^{in_mx} e^{-in_k\theta} e^{-in_k x} \ \frac{d\theta}{2\pi} \right) \frac{dx}{2\pi}, \\
&= \sum_{m=1}^\infty a_m r^{n_m} \left( \int_{a}^{b} e^{i(n_m - n_k)\theta}\ \frac{d\theta}{2\pi} \right) \left( \int_{-\pi}^\pi e^{i(n_m - n_k)x} \frac{dx}{2\pi} \right).
\end{align*}
We now evaluate the integral over the parameter $x$. Because $(n_k)$ is a strictly increasing sequence of positive integers, the orthogonality of characters on $\T$ implies that the integral $\int_{-\pi}^\pi e^{i(n_m - n_k)x} \frac{dx}{2\pi}$ evaluates to $1$ if $m = k$, and $0$ if $m \neq k$. Thus, 
\[
\int_{-\pi}^\pi \widehat{\iota_x}(n_k) e^{-in_k x} \frac{dx}{2\pi} = a_k r^{n_k} \left( \int_{a}^{b} 1\ \frac{d\theta}{2\pi} \right) \cdot 1 = a_k r^{n_k} \frac{b-a}{2\pi}.
\]
Isolating $a_k$ yields the desired equation.
\end{proof}

The \emph{total variation (norm)} of a Radon measure $\nu\in M(\T)$ is $\|\nu\|_{M(\T)}:=|\nu|(\T)$. The identity in Proposition~\ref{coeffProp} for the coefficients of the power series for $f$ suggests a natural bound in terms of the total variation norm of the measure $\iota_x$. More precisely, we have the following.

\begin{corollary}\label{CoeffCor}
If $1$ is a weak regular point of $f$, then for each positive integer $k$,
\[
|a_k| \le \frac{2\pi}{b-a} \sup_{e^{ix}\in\T} \|\iota_x\|_{M(\mathbb{T})}=\frac{2\pi}{b-a} \sup_{e^{ix}\in\T}\int_{a+x}^{b+x}|f(re^{i\vartheta})|\frac{d\vartheta}{2\pi}.
\]
\end{corollary}

\begin{proof} The equation in the desired conclusion may be obtained by using the definition of the measure $\iota_x$ and then some change of variable, so what remains to be shown is the desired inequality. Proceeding from the identity in Proposition~\ref{coeffProp}, $|a_k| \le \frac{2\pi}{r^{n_k}(b-a)} \int_{-\pi}^\pi \left| \widehat{\iota_x}(n_k) \right| \frac{dx}{2\pi}$.
Since the Fourier-Stieltjes transform of a Radon measure is bounded by its total variation norm, meaning $\left|\widehat{\iota_x}(n_k)\right| \le \|\iota_x\|_{M(\T)}$ for every $x$, we further have
\begin{align*}
|a_k| &\le \frac{2\pi}{r^{n_k}(b-a)} \int_{-\pi}^\pi \|\iota_x\|_{M(\T)} \frac{dx}{2\pi}, \\
&\le \frac{2\pi}{r^{n_k}(b-a)} \left( \sup_{e^{ix}\in\T} \|\iota_x\|_{M(\T)} \right) \int_{-\pi}^\pi \frac{dx}{2\pi}= \frac{2\pi}{r^{n_k}(b-a)} \sup_{e^{ix}\in\T} \|\iota_x\|_{M(\T)}.
\end{align*}
Taking the limit as $r \to 1^-$ results to the desired inequality.
\end{proof}

The upper bound given in Corollary~\ref{CoeffCor} has in the denominator the length $b-a$ of the underlying interval $(a,b)\sub(-\delta,\delta)$, which was described at the beginning of this section. Conceivably, choosing $(a,b)$ with a small enough length does not improve the aforementioned upper bound. However, an additional sufficient condition may be introduced such that, utilizing the compactness of $\T$, an absolute upper bound may be obtained, and thus, we have the following.

\begin{lemma}\label{breConverseLem} If every point along the unit circle $\T$ is a weak regular point of $f$, then the sequence $(a_k)$ is a bounded sequence of complex numbers.
\end{lemma}
\begin{proof} Let $\gamma:[-\pi,\pi)\into\T$ be the function $t\mapsto e^{it}$. In the beginning of this section, Section~\ref{CoeffSec}, we defined the arc $U(1)$ for the weak regular point $1$ of $f$. We generalize this by the assignment $e^{i\alpha}\mapsto U(e^{i\alpha})$ such that, for the weak regular point $e^{i\alpha}\in\T$ of $f$, the function $z\mapsto f(e^{i\alpha}z)$ has $1$ as a weak regular point, and by Corollary~\ref{CoeffCor}, for any positive integer $k$,
\begin{flalign}
    && |a_k| &\le \frac{2\pi}{b-a} \sup_{e^{ix}\in\T}\int_{a+x}^{b+x}|f(re^{i(\vartheta+\alpha)})|\frac{d\vartheta}{2\pi},&((a,b)\sub\gamma^{-1}[U(e^{i\alpha})]).\label{radialbound1}
\end{flalign}
The collection $\{U(e^{i\alpha})\  :\  e^{i\alpha}\in\T\}$ is an open cover of the compact space $\T$, so there exist finitely many $J_1:=U(e^{i\alpha_1})$, $J_2:=U(e^{i\alpha_2})$, $\ldots$ , $J_N:=U(e^{i\alpha_N})$ that cover $\T$. Thus, each $e^{ix}\in\T$ is in some $J_m$ where $m\in\{1,2,\ldots,N\}$, and this arc $J_m$ contains the weak regular point $e^{i\alpha_m}$ of $f$, which means that
\begin{equation}
K_m:=\sup_{0<r<1}\int_{e^{i\theta}\in J_m}|f(re^{i\theta})|\ \frac{d\theta}{2\pi}<\infty.\label{MdefProof}
\end{equation}
For each $m\in\{1,2,\ldots,N\}$, a nonempty interval $(u_m,v_m)\sub\gamma^{-1}[J_m]$ may be chosen that does not contain $-1$, so that the integral in \eqref{radialbound1}, with the limits of integration changed, to be from $u_m+x$ to $v_m+x$, need not be split into two integrals. Since there are finitely many intervals $(u_m,v_m)$,
\begin{eqnarray}
    \rho:=\max_{m\in\{1,2,\ldots,N\}}\frac{1}{v_m-u_m}<\infty.\label{denomMax}
\end{eqnarray}
Using \eqref{radialbound1}, 
\[
\int_{-\pi}^\pi |f(re^{i\phi})| \frac{d\phi}{2\pi} \leq\sum_{m=1}^N \int_{e^{i\theta}\in J_m} |f(re^{i\theta})| \frac{d\theta}{2\pi}\le \sum_{m=1}^N K_m  < \infty,
\]
where the left-most member, by the change of variable $\phi = \vartheta + \alpha_m$, is at least  $\int_{u_m+x}^{v_m+x}|f(re^{i(\vartheta+\alpha_m)})|\frac{d\vartheta}{2\pi}$. Taking the supremum over all $e^{ix}\in\T$,
\begin{equation}
\sup_{e^{ix}\in\T}\int_{u_m+x}^{v_m+x}|f(re^{i(\vartheta+\alpha_m)})|\frac{d\vartheta}{2\pi}\leq \sum_{m=1}^N K_m <\infty.\label{arcsup}
\end{equation}
From \eqref{radialbound1},\eqref{denomMax},\eqref{arcsup},
\[
|a_k|\leq 2\pi\rho\sum_{m=1}^NK_m<\infty.\qedhere
\]
\end{proof}
If $(a_k)$ is a bounded sequence of complex numbers, then by \cite[Theorem~1.6]{bre11}, if $(a_k)$ has a subsequence satisfying certain conditions, then any element of $\T$ is not a weak regular point of $f$. Since $\T$ is nonempty, there is indeed at least one such point. This proves that the converse of Lemma~\ref{breConverseLem} is false. The results in this paper, however, are based on Lemma~\ref{breConverseLem}, and we hope that this gives the reader a better insight on what the results in this paper contribute to the study of strong natural boundaries.

\section{An integral bound for an arbitrary arc}\label{TuranSec}

For each positive integer $K$, and given $r\in(0,1)$ and $e^{i\theta}\in\T$, we define $T_{K}(re^{i\theta}):=\sum_{k=1}^K a_kr^{n_k}e^{in_k\theta}$. A key component of the proof of Lemma~\ref{NTuranLem} below, which is the key result in this section, is the use of the following extension of the Tur\'an Lemma \cite[Corollary~18.5.VI]{tur84} that relates two finite polynomials. Since what we need is the version of the Tur\'an Lemma for integrals of polynomials, we shall be using instead the following extension of the Tur\'an Lemma for integrals.

\begin{lemma}[Nazarov-Tur\'an Lemma {\cite[Section~3]{naz00}}]\label{TuranLem} For a given choice of $\delta>0$, there exists $A>0$ such that for any positive integer $K$, and any $r\in(0,1)$,
\begin{equation}
\int_{-\pi}^{\pi} |T_{K}(re^{i\theta})| \frac{d\theta}{2\pi} \le \left(\frac{A\pi}{\delta}\right)^{K-1} \int_{-\delta}^{\delta} |T_{K}(re^{i\theta})| \frac{d\theta}{2\pi}.
\end{equation}
\end{lemma}

As cited above, the version of the Nazarov-Tur\'an Lemma that we shall be using is from the paper \cite{naz00}, and so, the interested reader may find in there a proof.

\begin{lemma}\label{NTuranLem} Let $e^{it}\in\T$, and consider an interval $(-\lambda,\lambda)$ such that\linebreak $t\in(-\lambda,\lambda)\sub[-\pi,\pi)$.  For a given choice of $\delta>0$, there exists $A>0$ such that for any $\Gamma>0$, there exists a positive integer $K_0$ such that for any integer $K\geq K_0$, and any $r\in(0,1)$,
\begin{eqnarray}
    \int_{t-\lambda}^{t+\lambda}\left|f(re^{i\theta})\right|\frac{d\theta}{2\pi} &\leq&1+\left(\frac{A\pi}{\delta}\right)^{\lfloor\Gamma K\rfloor}\lpar1+\int_{-\delta}^{\delta} \left|f(re^{i\theta})\right| \frac{d\theta}{2\pi}\rpar. \label{TuranPartial3}
\end{eqnarray}
\end{lemma}
\begin{proof}  Given $e^{it}\in\T$, given $\lambda>0$ as described in the statement, and given $\delta>0$, let $A>0$ be the constant described in Lemma~\ref{TuranLem}. If $0<r<1$, then for any $\theta\in(-\lambda,\lambda)$, the number $re^{i\theta}$ is in the disk $B(0,1)$ where the power series representation for $f$ exists. Using the fact that a convergent series is a sequence limit of partial sums, for each $\Gamma>0$, there exists a positive integer $K_0$ such that for any integer $K\geq K_0$, 
\[\left|f(re^{i\theta})-T_{\lfloor\Gamma K\rfloor+1}(re^{i\theta})\right|=\left|\sum_{k\geq \lfloor\Gamma K\rfloor +2}^\infty a_kr^{n_k}e^{in_k\theta}\right|<1.
\]
By a routine use of the Reverse Triangle Inequality, \[\left|f(re^{i\theta})\right|\leq 1+ |T_{\lfloor\Gamma K\rfloor+1}(re^{i\theta})|,\] both sides of which, we integrate from $-\lambda$ to $\lambda$ with respect to Lebesgue measure, and using Lemma~\ref{TuranLem},
\begin{eqnarray}
    \int_{t-\lambda}^{t+\lambda}\left|f(re^{i\theta})\right|\frac{d\theta}{2\pi} &\leq& \frac{\lambda}{\pi}+\int_{t-\lambda}^{t+\lambda} |T_{\lfloor\Gamma K\rfloor+1}(re^{i\theta})|\frac{d\theta}{2\pi},\nonumber\\
    &\leq  & \frac{\lambda}{\pi}+\int_{-\pi}^{\pi} |T_{\lfloor\Gamma K\rfloor+1}(re^{i\theta})|\frac{d\theta}{2\pi},\nonumber\\
    &\leq&\frac{\lambda}{\pi}+\left(\frac{A\pi}{\delta}\right)^{\lfloor\Gamma K\rfloor} \int_{-\delta}^{\delta} |T_{\lfloor\Gamma K\rfloor+1}(re^{i\theta})| \frac{d\theta}{2\pi}.\label{TuranPartial}
\end{eqnarray} Using the Triangle Inequality on \[ T_{\lfloor\Gamma K\rfloor+1}(re^{i\theta})=f(re^{i\theta})-\sum_{k\geq \lfloor\Gamma K\rfloor +2}^\infty a_kr^{n_k}e^{in_k\theta},\]
we obtain the inequalities
\[|T_{\lfloor \Gamma K\rfloor+1}(re^{i\theta})|\leq \left|f(re^{i\theta})\right|+\left|\sum_{k\geq \lfloor\Gamma K\rfloor +2}^\infty a_kr^{n_k}e^{in_k\theta}\right|< \left|f(re^{i\theta})\right|+1,
\]
so, by routine calculations, \eqref{TuranPartial} further becomes 
\begin{eqnarray}
    \int_{t-\lambda}^{t+\lambda}\left|f(re^{i\theta})\right|\frac{d\theta}{2\pi} &\leq&\frac{\lambda}{\pi}+\left(\frac{A\pi}{\delta}\right)^{\lfloor\Gamma K\rfloor}\lpar\frac{\delta}{\pi}+\int_{-\delta}^{\delta} \left|f(re^{i\theta})\right| \frac{d\theta}{2\pi}\rpar,\nonumber
\end{eqnarray}
but since the half-arc lengths $\lambda$ and $\delta$ cannot exceed $\pi$, the quantities $\frac{\lambda}{\pi}$ and $\frac{\delta}{\pi}$ cannot exceed $1$, and the desired inequality follows.
\end{proof}

\section{Gap conditions on exponents}\label{GapSec}

The importance of the inequality \eqref{TuranPartial3} in Lemma~\ref{NTuranLem} is that the integral, or more precisely the $L^1$ norm, along an arc containing an arbitrary $e^{it}\in\T$ is bounded above by some quantity involving the corresponding integral for the arc $I_\delta$ that contains $1$. Under an assumption that $1$ is a weak regular point, the inequality takes us close to the conclusion that the arbitrary $e^{it}\in\T$ is also a weak regular point, if not for the main obstacle, which is the quantity $\left(\frac{A\pi}{\delta}\right)^{\lfloor\Gamma K\rfloor}$, because of which, as the length of $I_\delta$ is made to be small, the upper bound in \eqref{TuranPartial3} grows without bound. As shall be explored in this section, the obstacle may be overcome if the sequence $(n_k)$ of exponents is lacunary, or has (increasing) gaps. This also echoes the property of lacunary series that it exhibits uniform ``behavior'' in all directions around the unit circle. We explore this behavior for, in our opinion, the two biggest classes of lacunary sequences, which are those that have zero harmonic density, and those that have zero density, both of which we shall define shortly. The former encompasses the Hadamard Gap Theorem, while the latter is the sufficient condition for the Fabry Gap Theorem, and is well-known to be the weakest possible sufficient condition for a natural boundary \cite[p.~102]{erd45}. Thus, the class of sets of zero density is the largest possible for series that have a natural boundary. However, the actual relationship between the two classes is unclear. The terminology we are using here is from \cite[Chapter~10]{gra13}. It is known that there is a set (of integers) of zero density that does not have zero harmonic density \cite[p.~186]{gra13}, but it is not known whether there is a set of zero harmonic density that is not of zero density \cite[Open Problem~13, p.~249]{gra13}. Whether the answer is in the affirmative or not, still, we are of the opinion that considering these two classes of sets of integers is as exhaustive as may be done at this point.

The sequence $(n_k)$ (or the set of all terms of this sequence) is said to have \emph{zero harmonic density} if for each open subset $U$ of $\T$ and each measure $\varphi\in M(\T)$, there exists a measure $\nu\in M(\T)$, concentrated on $U$ (meaning $\nu(\T\setdiff U)=0$), such that for any index $k$, we have $\widehat{\nu}(n_k)=\widehat{\varphi}(n_k)$. Some important classes of sets with zero harmonic density are Hadamard sets (see, for instance, \cite[Chapter~1]{gra13}), related  to the Hadamard Gap Theorem, and Sidon sets, which are relevant in the harmonic analysis topic of interpolation on compact groups \cite[Chapter~6]{gra13}.

Given $e^{it}\in\T$, by the \emph{evaluation functional $\Delta_t$}, we mean the nonnegative linear functional on $C(\T)$ defined by $\Delta_t(f)=f(e^{it})$, the dual of which is the measure $\delta_t\in M(\T)$ called the \emph{point mass measure} at $t$. That is, $f\in C(\T)$ implies $\Delta_t(f)=\int_\T f\  d\delta_t$. The corresponding Fourier-Stieltjes transform is $\widehat{\delta_t}(n)=e^{-int}$ for any character $n:e^{it}\mapsto e^{int}$ of $\T$.

In the succeeding proofs, an assumption of $1$ being a weak regular point entails the underlying arc $I_\delta$ and the supremum \eqref{Mdef2} as described in Section~\ref{PrelimSec}.

\begin{lemma}\label{zhdLem} If $(n_k)$ has zero harmonic density and if $1$ is a weak regular point of $f$, then so is any element of $\T$.
\end{lemma}
\begin{proof} Let $e^{it}\in\T$, and let $J_\delta$ be the image under $s\mapsto e^{is}$ of the interval $(-\frac{\delta}{3},\frac{\delta}{3})$, so $1\in J_\delta\sub I_\delta$. By zero harmonic density, for the open set $J_\delta\sub\T$ and the point mass measure $\delta_t$, there exists $\nu=\nu_t\in M(\T)$, concentrated on $J_\delta$, such that  
\begin{flalign}
    && \widehat{\nu}(n_k)&=e^{-in_kt}=\widehat{\delta_t}(n_k),&(k=1,2,\ldots).\label{zhdFStransform}
\end{flalign}
Define $g$ for each $z=re^{i\theta}\in B(0,1)$ (hence, $0<r<1$) by $g(z)=\int_{J_\delta} f(e^{-i\phi}z)\  d\nu(e^{i\phi})$. Using the power series expansion for $f$,
\begin{eqnarray}
    g(z) =  \int_{J_\delta} \left( \sum_{k=1}^\infty a_k z^{n_k} e^{-in_k\phi} \right) d\nu(e^{i\phi}).\label{zhdDCT}
\end{eqnarray}
Because $0<r<1$, by a routine use of the Root Test, $K_r:=\sum_{k=1}^\infty |a_k|r^{n_k}$ converges, and we define $G$ as the constant function $G(e^{i\phi}):=K_r$. If, for each positive integer $m$, we let $f_m(e^{i\phi}):=\sum_{k=1}^ma_kz^{n_k}e^{-in_k\phi}$, then $|f_m(e^{i\phi})|\leq |G(e^{i\phi})|$. Using the Lebesgue Dominated Convergence Theorem on the measure $\nu$, the series summation and integration in \eqref{zhdDCT} may be interchanged. Thus,
\begin{equation}
    g(z)=\sum_{k=1}^\infty a_kz^{n_k}\int_{J_\delta} e^{-in_k\phi}\  d\nu(e^{i\phi}),\label{gzDef}
\end{equation}
the integral in which, because $\nu$ is concentrated on $J_\delta$, is the same as the integral over all of $\T$. With further use of \eqref{zhdFStransform}, we obtain
\begin{eqnarray}
    g(z)&=&\sum_{k=1}^\infty a_kz^{n_k}\int_{\T} e^{-in_k\phi}\  d\nu(e^{i\phi})=\sum_{k=1}^\infty a_kz^{n_k}\widehat{\nu}(n_k),\nonumber\\
    &=&\sum_{k=1}^\infty a_kz^{n_k}e^{-in_kt},\nonumber\\
    g(z)&=& f(e^{-it}z),\nonumber
\end{eqnarray}
and since the mapping $z \mapsto e^{-it}z$ is a bijection on $B(0,1)$, 
\begin{flalign}
    && g(e^{it}z) &= f(z), &(z\in B(0,1)).\label{reverseG}
\end{flalign}
Integrating $|f(re^{i\psi})|$ over the interval $\left(-t-\frac{\delta}{3}, -t+\frac{\delta}{3}\right)$ with respect to Lebesgue measure, followed by using \eqref{reverseG}, and then performing the change of variable $\theta=\psi+t$, we obtain
\begin{eqnarray}
\int_{-t-\frac{\delta}{3}}^{-t+\frac{\delta}{3}}|f(re^{i\psi})|\ \frac{d\psi}{2\pi} &=& \int_{-t-\frac{\delta}{3}}^{-t+\frac{\delta}{3}}|g(re^{i(\psi+t)})|\ \frac{d\psi}{2\pi} = \int_{-\frac{\delta}{3}}^{\frac{\delta}{3}}|g(re^{i\theta})|\frac{d\theta}{2\pi},\nonumber\\
&=&\int_{-\frac{\delta}{3}}^{\frac{\delta}{3}} \left|\int_{J_\delta}f(re^{i(\theta-\phi)})\  d\nu(e^{i\phi})\right|\frac{d\theta}{2\pi},\nonumber\\
&\leq & \int_{-\frac{\delta}{3}}^{\frac{\delta}{3}} \left( \int_{J_\delta} |f(re^{i(\theta-\phi)})|\ d|\nu|(e^{i\phi}) \right) \frac{d\theta}{2\pi},
\end{eqnarray}
where the last integral is with respect to the product measure $|\nu|\times \LebT$, where $\LebT$ is Lebesgue measure on $\T$. By Fubini's Theorem, we further have
\[
\int_{-t-\frac{\delta}{3}}^{-t+\frac{\delta}{3}}|f(re^{i\psi})|\ \frac{d\psi}{2\pi} \le \int_{J_\delta} \left( \int_{-\frac{\delta}{3}}^{\frac{\delta}{3}} |f(re^{i(\theta-\phi)})|\ \frac{d\theta}{2\pi} \right) d|\nu|(e^{i\phi}).
\]
On the inner integral, we perform the change of variable $\omega=\theta-\phi$ to obtain
\begin{equation}
\int_{-t-\frac{\delta}{3}}^{-t+\frac{\delta}{3}}|f(re^{i\psi})|\ \frac{d\psi}{2\pi} \le \int_{J_\delta} \left( \int_{-\frac{\delta}{3}-\phi}^{\frac{\delta}{3}-\phi} |f(re^{i\omega})|\ \frac{d\omega}{2\pi} \right) d|\nu|(e^{i\phi}).\label{InDeltaIneq}
\end{equation}
From the definition of $g$, we have $-\frac{\delta}{3}<\phi<\frac{\delta}{3}$, which, by routine computations, implies
\[
\left[ -\frac{\delta}{3}-\phi, \,\, \frac{\delta}{3}-\phi \right] \sub \left( -\delta, \,\, \frac{2\delta}{3} \right)\sub (-\delta,\delta).\label{zhdInsidedelta}
\]
Thus, $\int_{-\frac{\delta}{3}-\phi}^{\frac{\delta}{3}-\phi} |f(re^{i\omega})|\ \frac{d\omega}{2\pi} \le \int_{-\delta}^{\delta} |f(re^{i\omega})|\ \frac{d\omega}{2\pi}$, so \eqref{InDeltaIneq} further becomes
\begin{equation}
\int_{-t-\frac{\delta}{3}}^{-t+\frac{\delta}{3}}|f(re^{i\psi})|\ \frac{d\psi}{2\pi} \le \int_{J_\delta} \left( \int_{-\delta}^{\delta} |f(re^{i\omega})|\ \frac{d\omega}{2\pi} \right) d|\nu|(e^{i\phi}).\label{InDeltaIneq2}
\end{equation}
If $M$ is the supremum from \eqref{Mdef2}, then using the fact that $\nu$ is concentrated on $J_\delta$, the inequality \eqref{InDeltaIneq2} further becomes
\[
\int_{-t-\frac{\delta}{3}}^{-t+\frac{\delta}{3}}|f(re^{i\psi})|\ \frac{d\psi}{2\pi} \le \int_{J_\delta} M\  d|\nu|(e^{i\phi})=M\int_\T\  d|\nu|=M\|\nu\|_{M(\T)},
\]
where $\|\nu\|_{M(\T)}$ is finite because $\nu$ is a Radon measure on the compact space $\T$. Thus, $\sup_{0<r<1}\int_{-t-\frac{\delta}{3}}^{-t+\frac{\delta}{3}}|f(re^{i\psi})|\ \frac{d\psi}{2\pi}<\infty$, so $e^{-it}$ is a weak regular point of $f$. But since $e^{it}\in\T$ is arbitrary, the desired conclusion follows.
\end{proof}

The above proof follows the structure of that for \cite[Proposition~10.4.7]{gra13}, and the statement of the lemma echoes (for zero harmonic density) what has been discussed at the beginning of this section, Section~\ref{GapSec}, of what we intend to show for sets of zero density, that utilizes the behavior that a gap series has uniformly at any direction around the unit circle. We now proceed with the specific details on sets of zero density.

A set $E$ of integers is said to have \emph{zero density} if ${\displaystyle\limsup_{n\rightarrow\infty}}\frac{|E\cap[-n,n]|}{2n+1}=0$. Since we shall be considering only the case when $E$ is the collection of all terms of the sequence $(n_k)$ of positive integers, this condition is equivalent to $\klim\frac{k}{n_k}=0$, the well-known sufficient condition for a natural boundary, in what is called the Fabry Gap Theorem.

We now give the analog of Lemma~\ref{zhdLem} for sets of zero density. The proof is based on the technique in \cite[Section~5.3]{mon94}, which, in connection to the explanation at the beginning of this section, Section~\ref{GapSec}, shows us how to overcome the obstacle that the quantity $\left(\frac{A\pi}{\delta}\right)^{\lfloor\Gamma K\rfloor}$ in the inequality \eqref{TuranPartial3} in Lemma~\ref{NTuranLem} presents. As the main technique used in \cite[Section~5.3]{mon94} to prove the Fabry Gap Theorem is the Tur\'an Lemma that relates two finite polynomials, but what we needed here is the corresponding lemma for integrals of such polynomials, the appropriate machinery is the extension of the Tur\'an Lemma into the Nazarov-Tur\'an Lemma, Lemma~\ref{TuranLem}, that was used in the proof of Lemma~\ref{NTuranLem}.

\begin{lemma}\label{zdLem} If $(n_k)$ has zero density and if $1$ is a weak regular point of $f$, then so is any element of $\T$.
\end{lemma}
\begin{proof} For each $x>0$, let $\sigma(x)$ be the number of exponents $n_k$ that are at most $x$.
If $\klim\frac{k}{n_k}=0$, then by a routine argument, ${\displaystyle\lim_{x\rightarrow\infty}}\frac{\sigma(x)}{x}=0$, so, under all the hypotheses in Lemma~\ref{NTuranLem}, a small enough $\Gamma$ may be chosen such that for any integer $K\geq K_0$, the quantity $\lfloor\Gamma K\rfloor=\sigma(\Gamma K)$ is bounded by $\Gamma K$. Thus, $\left(\frac{A\pi}{\delta}\right)^{\lfloor\Gamma K\rfloor}\into 1$ as $K\into\infty$, and the inequality \eqref{TuranPartial3} in the conclusion of Lemma~\ref{NTuranLem} further becomes
$
    \int_{t-\lambda}^{t+\lambda}\left|f(re^{i\theta})\right|\frac{d\theta}{2\pi} \leq 2+\int_{-\delta}^{\delta} \left|f(re^{i\theta})\right| \frac{d\theta}{2\pi}.
$
Under the assumption that $1$ is a weak regular point of $f$, if $M$ is the supremum in \eqref{Mdef2}, then taking the suprema, over all $r\in(0,1)$, of both sides of the previous inequality, we have $\sup_{0<r<1}\int_{t-\lambda}^{t+\lambda}\left|f(re^{i\theta})\right|\frac{d\theta}{2\pi}\leq 2+M<\infty$. Therefore, $e^{it}$ is a weak regular point of $f$.
\end{proof}

\section{Summary and further directions}\label{ThmSec}

Under the assumption that the coefficient sequence $(n_k)$ has zero harmonic density or zero density, if $\T$ has an element that is a weak regular point of $f$, (which, according to our explanation in Section~\ref{PrelimSec}, may be assumed without loss of generality to be $1$), then by Lemmas~\ref{zhdLem},\ref{zdLem}, and then by Lemma~\ref{breConverseLem}, the sequence $(a_k)$ is bounded. By contraposition, and some rearrangement of the ensuing sufficient conditions, we obtain our main result.
\begin{theorem} Let $(a_k)$ be a sequence of complex numbers, and let $(n_k)$ be a strictly increasing sequence of positive integers, such that the power series $f(z)=\sum_{k=1}^\infty a_kz^{n_k}$ converges if $|z|<1$. If the sequence $(n_k)$ has either zero harmonic density or zero density, and if $\sup_k|a_k|=\infty$, then $\T$ is a strong natural boundary of $f$.
\end{theorem}
If $a_k=k=n_k$ for all $k$, then $(n_k)$ has neither zero harmonic density nor zero density, with $\sup_k|a_k|=\infty$, but $-1$ is a regular, and hence a weak regular, point of $f$, so $\T$ is not a strong natural boundary of $f$. This very simple example shows us that the assumption, of having either zero harmonic density or zero density in the theorem, is substantial, or that the absence of it does not reduce the statement to a triviality. For another example, as pointed out in \cite[p.~4905]{bre11}, the series $f(z)=\sum_{k=1}^\infty\frac{z^{n!}}{(n!)^n}$ \cite[Equation~(1.9)]{bre11} does not have $\T$ as a strong natural boundary, with the exponents forming a Hadamard set (which has zero harmonic density) and also satisfying the Fabry gap condition (which means zero density), and the sequence of coefficients is bounded. This means that the condition of having zero harmonic density \emph{and} (and hence \emph{or}) zero density is not sufficient for having a strong natural boundary, without the assumption that $\sup_k|a_k|=\infty$.

The connection of this topic to the theory of Hardy spaces and Smirnov spaces \cite{dur70, hof62} on a domain (an open connected subset of the complex plane) cannot be denied. For instance, if $1$ is a weak regular point of $f$, and if we consider the reflection of the arc $I_\delta$ with respect to the vertical line connecting the intersections of the endpoints of the arc with the unit circle $\T$, then the domain $\Omega$ enclosed by the topological closure of the union of the arc and its reflection has an associated Smirnov class $E^{1}(\Omega)$ of functions to which the restriction of $f$ belongs. See \cite[Section~10.1]{dur70}. The domain $\Omega$ is ``lens-shaped'' or convex. Under a gap assumption, such as zero harmonic or zero density, the said Smirnov class may be replicated in all directions around the unit circle. Thus, either of the zero density conditions implies that $f$ is in the Hardy class $H^1(\T)$. We propose continuations of this study that explore how the relationships between the function spaces $H^{1}(\T)$ and $E^{1}(\Omega)$, where $\Omega$ is some lens-shaped region constructed from a weak regular point, result to conditions, whether sufficient or necessary, for the gap conditions on the power series exponents. The study of these function spaces, as articulated in \cite[p.~vii]{hof62}, highlights an interplay between functional analysis and complex analysis, and by our proposal to study the effect of gap conditions, we also introduce harmonic analysis to such a study.

\section*{Ethical approval} This is not applicable to research work in pure mathematics.

\section*{Funding} The author was supported by the Research Grants Management Office of De La Salle University, Taft Ave., Manila, Philippines, with grant no.: 02FR1TAY24-1TAY25.

\section*{Availability of data and materials} Studies in pure mathematics do not involve data sets, and hence a declaration on availability of data and materials is not applicable.


\begin{thebibliography}{19}

\bibitem{bre11}
Breuer, J., Simon, B.: Natural boundaries and spectral theory. Adv. Math. \textbf{226}(6), 4902--4920 (2011)

\bibitem{dos26}
Dostoglou, S., Valettas, P.: A probabilistic approach to strong natural boundaries. Mathematische Zeitschrift \textbf{313}(3), 48 (2026)

\bibitem{dur70}
Duren, P.L.: Theory of $H^p$ Spaces. Pure and Applied Mathematics, vol. 38. Academic Press, New York-London (1970)

\bibitem{erd45}
Erd{\H{o}}s, P.: Note on the converse of Fabry's gap theorem. Trans. Am. Math. Soc. \textbf{57}, 102--104 (1945)

\bibitem{gra13}
Graham, C.C., Hare, K.E.: Interpolation and Sidon sets for compact groups. Springer, New York, NY (2013)

\bibitem{hof62}
Hoffman, K.: Banach Spaces of Analytic Functions. Prentice-Hall Series in Modern Analysis. Prentice-Hall, Inc., Englewood Cliffs, N.J. (1962)

\bibitem{mon94}
Montgomery, H.L.: Ten lectures on the interface between analytic number theory and harmonic analysis. American Mathematical Society, Providence, RI (1994)

\bibitem{naz00}
Nazarov, F.: Complete version of Turan's lemma for trigonometric polynomials on the unit circumference. In: Complex Analysis, Operators, and Related Topics, vol. 113, pp. 239--246. Birkh\"{a}user, Basel (2000)

\bibitem{sha68}
Ross, W.T., Shapiro, H.S.: Generalized analytic continuation. University Lecture Series, vol. 25. American Mathematical Society, Providence, RI (2002)

\bibitem{seg08}
Segal, S.L.: Nine introductions in complex analysis, 2nd revised ed. North-Holland Mathematics Studies, vol. 208. Elsevier, Amsterdam (2008)



\bibitem{tur84}
Tur{\'a}n, P.: On a new method of analysis and its applications. John Wiley \& Sons, New York (1984)

\end{thebibliography}
\end{document}